\documentclass{amsart}
\usepackage{math}
\usepackage{lineno,tikz,mathabx,mathrsfs,mathtools,amsrefs}
\usetikzlibrary{cd,quotes,angles,decorations,arrows}
\usepackage[letterpaper,margin=35mm]{geometry}

\newtheorem{mainthm}{Theorem}

\newtheorem{maincor}[mainthm]{Corollary}
\newtheorem{mainconj}[mainthm]{Conjecture}
\newtheorem{mainprop}[mainthm]{Proposition}

\newcommand\Mod{\operatorname{Mod}}
\newcommand\PSL{\operatorname{PSL}}
\newcommand\Rat{{\mathbf{Rat}}}
\newcommand\Pone{{\mathbb P^1}}
\newcommand\Def{{\mathcal D\!\mathit{ef}}}
\newcommand\Homeo{{\operatorname{Homeo}}}

\newcount\tmpcnta

\begin{document}
\title{Connectedness of a space of branched coverings with a periodic cycle}
\author{Laurent Bartholdi}
\address{Fachrichtung Mathematik+Informatik, Universität des Saarlandes}
\email{laurent.bartholdi@gmail.com}
\date{April 13th, 2022}
\begin{abstract}
  We prove the connectedness of the following locus: the space of degree-$d$ branched self-coverings of $S^2$ with two critical points of order $d$, one of which is $n$-periodic.  
\end{abstract}
\maketitle

\section{Introduction}
Consider the space $\mathscr M_d$ of degree-$d$ branched self-coverings of the sphere $S^2$. Much is known about its topology~\cite{segal:rational}, for example it is connected, and its fundamental group is $\mathbb Z/2d$. Interesting subspaces arise by imposing dynamical conditions; the one we will focus on in this article is
\[\mathscr P_{d,n}=\{f\colon S^2\righttoleftarrow: f\text{ has two critical points of order $d$, one of which has period exactly }n\}.\]
Our main result is:
\begin{mainthm}\label{thm:main}
  The space $\mathscr P_{d,n}$ is path connected.
\end{mainthm}
An equivalent statement is that any two maps in $\mathscr P_{d,n}$ are isotopic through maps within $\mathscr P_{d,n}$. In fact, the homotopy type of $\mathscr P_{d,n}$ can be determined quite explicitly (see Proposition~\ref{prop:homotopy}).

This purely topological statement has a complex analytic avatar: endow $S^2$ with its complex structure, now written $\Pone$. In this manner, $\Rat_d=\{f\in\C(z):\deg(f)=d\}$ embeds into $\mathscr M_d$. The locus $\mathscr P_{d,n}\cap\Rat_d$ is an important ``slice'' of parameter space, whose connectedness was asked by Milnor~\cite{milnor:quadratic}. In fact, the group of M\"obius transformations acts on $\Rat_d$ by conjugation, preserving the locus  $\mathscr P_{d,n}\cap\Rat_d$, so this question may be studied in the quotient space. Milnor proves that $\{f\in\Rat_d:f\text{ has two critical points of order }d\}/\PSL_2(\C)$ is isomorphic to $\C^2$, for instance by identifying the map $(\alpha z^d+\beta)/(\gamma z^d+\delta)$ with the pair $(\beta\gamma/(\alpha\delta-\beta\gamma),(\alpha^{d+1}\beta^{d-1}+\gamma^{d-1}\delta^{d+1})/(\alpha\delta-\beta\gamma)^{2d})$. Then the image of $\mathscr P_{d,n}\cap\Rat_d$ in $\C^2$ is an algebraic curve called $\mathbf{Per}_n^d(0)$, whose connectedness is a tantalizing open problem. Theorem~\ref{thm:main} should be seen as a solution to this problem in the topological context; since by~\cite{segal:rational} the inclusion of $\Rat_d$ in $\mathscr M_d$ is a homotopy equivalence up to dimension $d$, one may similarly hope:
\begin{mainconj}
  The inclusion $\mathscr P_{d,n}\cap\Rat_d\hookrightarrow\mathscr P_{d,n}$ is a homotopy equivalence up to dimension $d$.
\end{mainconj}

\subsection{Spaces of marked maps}
The general setting is a structure $\Pi=\langle A,B,C,F,\deg\rangle$ consisting of sets $A,B,C$ with $A\subset B$ and $A\subset C$, and maps $F\colon C\to B$ and $\deg\colon C\to\{1,2,\dots\}$. The corresponding space of marked maps is
\begin{multline*}
  \mathscr P_\Pi\coloneqq\{(f,b,c):b\colon B\hookrightarrow S^2,c\colon C\hookrightarrow S^2,f\colon(S^2,c(C))\to(S^2,b(B))\text{ branched covering},\\
  f\circ c=b\circ F,\;\deg_{c(x)}f=\deg(x),\;b\restriction A=c\restriction A,\;f\text{ ramifies only above }b(B)\}.
\end{multline*}
For example, setting $A=\{a_0,\dots,a_{n-1}\}$ and $B=A\cup\{v\}$ and $C=A\cup\{c\}$ with $\deg(a_0)=\deg(c)=d$ and $F(a_i)=a_{i+1\bmod n}$ and $F(c)=b$ specifies precisely the space $\mathscr P_{d,n}$ introduced above. The method of this article should serve to solve a variety of connectedness problems about loci $\mathscr P_\Pi$.

Thurston's theory of iteration of rational maps involves considering their topological counterparts. A branched covering $f$ is \emph{critically finite} if the forward orbit $P_f$ of its critical points is finite. Setting $A=B=C=P_f$ and $F=f\restriction P_f$ recovers the space of all maps with same post-critical behaviour as $f$, and his fundamental result implies that the connected component of $f$ is contractible and contains at most one holomorphic representative (with a combinatorial criterion to determine whether there is one), unless the map is double covered by a homothety on the torus. At the other extreme, $A=\emptyset$ amounts to considering branched coverings with no dynamical constraint, and subsumes the classical Hurwitz theory of coverings of surfaces~\cite{hurwitz:ramifiedsurfaces}.

By abuse of notation, we will think of $B,C$ both as abstract sets (used to define $\Pi$) and as variable subsets of $S^2$; and abbreviate
\[\mathscr P_\Pi=\{f\colon(S^2,C)\to(S^2,B):f\restriction A=F\restriction A,\deg_f=\deg\}.\]

Kahn, Firsova and Selinger~\cite{firsova-kahn-selinger:deformation} consider, following Mary Rees~\cite{rees:tr}, a space of maps modeled on Teichm\"uller space, which we can express as follows in our setting. Recall that the Teichm\"uller space of a marked sphere consists of all its complex structures; for $B\subset S^2$ we may define
\[\mathcal T_B=\{\tau\colon S^2\to\Pone\text{ homeomorphism}\}/{\sim},
\]
with $m\circ\tau\circ h\sim\tau$ for every M\"obius transformation $m$ and every homeomorphism $h$ of $(S^2,B)$ that is isotopic to the identity rel $B$; henceforth written $h\in\Homeo_0(S^2,B)$. Then \emph{Rees space} is
\[\mathscr R_\Pi=(\mathcal T_B\times\mathscr P_\Pi)/{\sim},
\]
with $(\tau\circ h,f)\sim(\tau,h\circ f)\sim(\tau,f\circ h')$ whenever $h\colon(S^2,B)\righttoleftarrow,h'\colon(S^2,C)\righttoleftarrow$ satisfy $h\circ f\simeq f\circ h'$ rel $C$ and $h\simeq h'$ rel $A$. Every choice of $f\in\mathscr P_\Pi$ yields a map $\mathcal T_B\to\mathscr R_\Pi$, which is a covering on its image, the group of deck transformations being the centralizer $\{h\in\mathbf{Mod}_B:h\circ f=f\circ h\text{ rel }A\}$ of $f$. Theorem~\ref{thm:main} can be rephrased as the statement
\begin{maincor}
  For $\Pi$ marking two order-$d$ critical points one of which in $n$-periodic, the space $\mathscr R_\Pi$ is a quotient of Teichm\"uller space $\mathcal T_B$ by a subgroup of $\Mod(B)$, so in particular is a $K(\pi,1)$ space for some $\pi\le\Mod(B)$.
\end{maincor}

There is a fibration $\mathscr P_\Pi\to\mathscr Q$, the space of configurations of $B\cup_A C$ in $S^2$. Let $\mathscr W$ denote a fibre, thought of as a space of branched coverings with specified marked points. The group $\Homeo_0(S^2,C)$ is contractible and acts freely on $\mathscr W$ by pre-composition, with discrete quotient. From the long exact sequence of homotopy groups,
\begin{mainprop}\label{prop:homotopy}
  The space $\mathscr W$ is homotopy equivalent to a discrete set. The space $\mathscr P_{d,n}$ is path connected; $\pi_1(\mathscr P_{d,n})\cong\ker(\pi_1(\mathscr Q)\twoheadrightarrow\pi_0(\mathscr W))$; and $\pi_q(\mathscr P_{d,n})\cong\pi_q(\mathscr Q)$ for all $q>1$. Furthermore, $\pi_1(\mathscr Q)$ is naturally identified with the equalizer of the two natural (forgetful) maps $\Mod(B)\to\Mod(A),\Mod(C)\to\Mod(A)$, and $\pi_0(\mathscr W)$ is naturally identified with $\Mod(C)\times\{1,\dots,d^{n-2}$.\qed
\end{mainprop}

Every $(\tau,f)\in\mathscr R_\Pi$ yields two complex structures on $(S^2,A)$: one locally given by $\tau$ and one by $\tau\circ f$ and denoted $\sigma_f(\tau)$. The \emph{rational Rees space} $\mathscr R_\Pi^\Rat$ is the locus at which these complex structures agree; so for $(\tau,f)\in\mathscr R_\Pi^\Rat$ we have a commutative diagram
\[\begin{tikzcd}
    (S^2,C)\arrow[r,"\tau"]\arrow[d,"f" left]\arrow[dr,phantom,"\ulcorner\sim",very near start] & \Pone\arrow[d,"r_{\tau,f}"]\\
    (S^2,B)\arrow[r,"\tau"] & \Pone
  \end{tikzcd}
\]
that commutes up to homotopy, for some rational map $r_{\tau,f}\in\Rat_d$, unique up to conjugation by a M\"obius transformation.

Let $\mathbf{MRat}$ denote the quotient of $\Rat=\C(z)$ by $\PSL_2(\C)$ acting under conjugation; and consider the locus
\[\mathbf{MRat}_\Pi=(\Rat\cap\mathscr P_\Pi)/\PSL_2(\C).\]
Note that every $f\in\mathbf{MRat}_\Pi$ has attached subsets $A,B,C\subset\Pone$, well-defined up to a M\"obius transformation. Every choice of $\tau\in\mathcal T_B$ gives a natural map $\mathscr P_\Pi\to\mathscr R_\Pi$, by $f\mapsto(\tau,f)$. The natural map $\mathscr R_\Pi^\Rat\to\mathbf{MRat}_\Pi$, given by $(\tau,f)\mapsto r_{\tau,f}$, is an isomorphism. There is finally a map $\mathbf{MRat}_\Pi\to\mathcal M_B$, the moduli space of $(S^2,B)$, given by $r_{\tau,f}\mapsto\tau\restriction B$. Consider, to complete the picture, Epstein's equalizer space: it is the space
\[\Def_{\Pi,f}=\{\tau\in\mathcal T_B:\tau\text{ and }\sigma_f(\tau)\text{ have the same image in }\mathcal T_A\}.\]
Epstein's transversality theory~\cite{epstein:transversality} shows that $\Def_{\Pi,f}$ is a submanifold of dimension $\#B-\#A-3$. In summary, we have
\[\begin{tikzcd}
    \Def_{\Pi,f}\arrow[r,hook]\arrow[d,"{\tau\mapsto(\tau,f)}" left] & \mathcal T_B\arrow[d]\\
    \makebox[0pt][r]{$\mathbf{MRat}_\Pi={}$}\mathscr R_\Pi^\Rat\arrow[r,hook]\arrow[d,"{(\tau,f)\mapsto\tau\restriction B}" left] & \mathscr R_\Pi\arrow[d]\\
\mathcal M_\Pi\arrow[r,hook] & \mathcal M_B
  \end{tikzcd}
\]
and our result shows that the top vertical maps are onto. The bottom left vertical map is a finite-to-one map~\cite{firsova-kahn-selinger:deformation}*{Lemma~4.5} and in many cases, in particular that of a degree-$d$ map with a marked $n$-cycle~\cite{firsova-kahn-selinger:deformation}*{Lemma~4.7} is actually a bijection.

Previous literature concentrated on connectivity and contractibility of $\Def_\Pi$. The arguments in~\cites{firsova-kahn-selinger:deformation,hironaka-koch:disconnected}, proving that $\Def_\Pi$ is disconnected or at least not contractible, rely on showing that the fundamental group of the punctured Riemann surface $\mathbf{MRat}_\Pi$ is not isomorphic to the equalizer $E\le\Mod(B)$.

\subsection{Sketch of proof}\label{ss:sketch}
We reduce Theorem~\ref{thm:main} to a discrete problem, by means of isotopy. 
\begin{enumerate}
\item Consider the fibration $\mathscr P_\Pi\to\mathscr Q$ and let $W$ be the collection of isotopy classes of a fibre. In our concrete situation, we imagine that sets $A,B,C$ of respective sizes $n,n+1,n+1$ are frozen on $S^2$, and we consider the collection $W$ of isotopy classes of maps $(S^2,C)\to(S^2,B)$ that permute $A$ cyclically. Thus $W$ is the collection of isotopy classes of maps $f\in\mathscr P_{d,n}$ with these fixed $A,B,C$.

\item There are two commuting actions on $W$, by the pure mapping class groups $\Mod(B)$ and $\Mod(C)$, respectively by post-composition and pre-composition. The action of $\Mod(C)$ is free with $d^{n-2}$ orbits.

\item There is a single degree-$d$ covering with two order-$d$ marked critical points; this is the ``Hurwitz problem'' for the data $(d,d)$. In fact, by~\cite{thom:polynome}*{Th\'eor\`eme (1)} this holds more generally for any degree-$d$ covering with a critical point of order $d$. It follows that there is a single $(\Mod(B)\times\Mod(C))$-orbit on $W$.

\item Denote respectively by $\varepsilon_B,\varepsilon_C$ the natural restriction maps $\Mod(B)\to\Mod(A),\Mod(C)\to\Mod(A)$. The subgroup $E=\{(g,h)\in\Mod(B)\times\Mod(C):\varepsilon_B(g)=\varepsilon_C(h)\}$ acts on $W$, and movements in $E$ preserve the cycle $A$. Connectedness of $\mathscr P_{d,n}$ is therefore equivalent to transitivity of $E$ on $W$.

\item To prove transitivity of $E$, fix a polynomial map $f\in\mathscr P_{d,n}$, thereby assuming $B=C$. We consider two subgroups of $E$: the diagonal $\Delta=\{(g,g):g\in\Mod(B)\}$, which acts on $W$ by conjugation, and $P=\ker(\varepsilon_B)$, the ``point pushes'' of the critical value, which act on $W$ by post-composition. It suffices to show that every element of $W$ may be written $p g f g^{-1}$ with $g\in\Mod(B)$ and $p\in P$.

\item Compute the ``lifting'' operation: let $L\le\Mod(B)$ denote the index-$d^{n-2}$ subgroup of liftable classes; namely all $h\in\Mod(B)$ such that there exists $\phi(h)\in\Mod(C)$ with $h\circ f\cong f\circ\phi(h)$. Enough images under homomorphism $\phi\colon L\to\Mod(C)$ may be written explicitly to show that all generators of $\Mod(B)$ may be obtained from $P$ via $\phi$ and conjugation.

\end{enumerate}

\subsection{Acknowledgments}
This work owes much to Jan Kiwi, who explained to me the relevance of group-theoretic considerations to connectedness questions in parameter spaces during an invitation to Santiago. I am very grateful to MSRI for providing an extremely stimulating environment, and to Tanya Firsova, Eko Hironaka, Jeremy Kahn, Sarah Koch and Curtis McMullen for lively discussions and generous explanations of their work.

\section{Dynamical bisets and mapping class bisets}
Let us consider the polynomial $f(z)=z^d+c$ for which the supporting rays have angles $\{1,2\}/(d^n-1)$, and encode $f$ group-theoretically. This means we let $A=\{a_0=a_n=0,a_1=c,a_2=c^d+c,\dots\}$ be the length-$n$ critical cycle of $f$; set $B=C=A\cup\{\infty\}$; fix a basepoint $*$ in $\C\setminus A$ near $\infty$; let $\pi=\pi_1(\C\setminus A,*)$ be the fundamental group; and choose ``lollipop'' generators $\gamma_1,\dots,\gamma_n,\gamma_\infty$ of $\pi$: the generator $\gamma_i$ follows the external ray with angle $d^i/(d^n-1)$ towards $a_i$, encircles it counterclockwise, and returns back to $*$, while the generator $\gamma_\infty$ is a small (on the sphere) counterclockwise loop around $\infty$. Note that one of the generators of $\pi$ is redundant, and we have
\[\pi=\langle \gamma_1,\dots,\gamma_n,\gamma_\infty\mid \gamma_\infty\gamma_n\cdots\gamma_1=1\rangle.\]

The ``iterated monodromy group'' theory~\cite{nekrashevych:ssg} exncodes $f$ as a \emph{biset}: a set $\mathfrak B(f)$ with two commuting $\pi$-actions:
\[\mathfrak B(f)=\{\alpha\colon[0,1]\to\C\setminus A:\alpha(0)=f(\alpha(1))=*\}/{\sim},\]
with the left action by pre-concatenation of a loop, and the right action by post-concatenation with the unique lift of a loop that starts where $\alpha$ ends.
The left action on $B_f$ is free with $d$ orbits; so $B_f$ may be written $\pi\times\{x_1,\dots,x_d\}$ by choosing a system of orbit representatives. A natural choice consists of short, counterclockwise paths around $\infty$ from $*$ to all its preimages; for concreteness if $*\gg0$ we may choose $x_j(t)=\exp(\log(*)(d+1-d t)/d+2\pi i j t/d)$.

It is then straightforward to trace paths and their lifts, so as to express the right action of $\pi$ on $B_f$; it is
\begin{xalignat*}{3}
  x_1\cdot\gamma_1&=\gamma_\infty\gamma_n\cdot x_2, &
  x_i\cdot\gamma_1&=x_{i+1}\text{ if }1<i<d, &
  x_d\cdot\gamma_1&=\gamma_\infty^{-1}\cdot x_1,\\
  x_1\cdot\gamma_{j+1}&=\gamma_j\cdot x_1, &
  x_i\cdot\gamma_{j+1}&=x_i\text{ if }1<i\le d\text{ and }1\le j\le n,\\
  x_1\cdot\gamma_\infty&=\gamma_\infty\cdot x_d, &
  x_{i+1}\cdot\gamma_\infty&=x_i\text{ if }1\le i<d.
\end{xalignat*}
This is the natural degree-$d$ generalization of the recursion defining the group $\mathfrak K(0^{n-1})$ from~\cite{bartholdi-n:mandelbrot1}; see also~\cite{poirier:portraits}.

Note that a different choice of orbit representatives $\{x_1',\dots,x_d'\}$ would give different formulas for the right action; the object $B_f$, considered up to isomorphism of $\pi$-$\pi$-bisets, is a complete invariant of $f$ but its presentation relies on choices.

Let us make this a bit more precise. A \emph{presentation} of a biset $B_{f'}$, for a map $f'\in W$, is a $d\times(n+1)$ matrix with in entry $(i,j)$ a pair $(g,k)\in\pi\times\{1,\dots,d\}$, describing the relation $x_i\cdot\gamma_j=g\cdot x_k$; here we use the convention $n+1=\infty$. A permutation of the $i$ and $k$ defines an isomorphic presentation (this amounts to reordering the orbit representatives); and for every choice of $g_1,\dots,g_d\in\pi$ the replacement of every $(g,k)$ at position $(i,j)$ with $(g_i g g_k^{-1},k)$ also defines an isomorphic presentation (this amounts to replacing the orbit representatives $(x_1,\dots,x_d)$ with $(g_1 x_1,\dots,g_d x_d)$).

Furthermore, the left and right actions on $W$ may be expressed in terms of these presentations: pre-composition by $\psi$ amounts to replacing each $(g,k)$ by $(\psi(g),k)$; while post-composition by $\phi$ amounts to using the table to rewrite $x_i\cdot\phi(\gamma_j)$ in the form $g\cdot x_k$ and recording the result in a new table.

Now, $\Mod(B)$ is generated by a collection of full Dehn twists between elements of $B$, and acts on $\pi$ by outer automorphisms. We choose as generators the $\tau_{i,j}$ for $1\le i<j\le\infty$ which, geometrically, push $a_i$ and $a_j$ closer while avoiding all other paths $\gamma_k$, and twist them fully around each other. The action of $\tau_{i,j}$ on $\pi$ may be written concretely as follows: set $\alpha=\gamma_{j-1}\gamma_{j-2}\cdots\gamma_{i+1}$; then
\[\tau_{i,j}(\gamma_i)=\gamma_i^{\alpha^{-1}\gamma_j \alpha\gamma_i},\qquad\tau_{i,j}(\gamma_j)=\gamma_j^{\alpha\gamma_i \alpha^{-1}},\qquad\tau_{i,j}(\gamma_k)=\gamma_k.\]
            
The subgroup of point pushes is $P=\langle\tau_{i,\infty}:1\le i\le n\rangle$, and its liftable elements include $\tau_{i,\infty}^d$ as well as $\tau_{1,\infty}$ and its conjugates.

Once all these choices are made, it is straightforward to check the following identities:
\begin{lem}\label{lem}
  In $W$ we have the identities
  \begin{align}
    \tau_{1,\infty}\cdot f&=f\cdot\tau_{n,\infty},\\
    \tau_{1,\infty}^{\tau_{i+1,\infty}}\cdot f&=f\cdot \tau_{i,\infty}\tau_{i,n}\tau_{n,\infty},\label{eq:1conj}\\
    \tau_{i+1,\infty}^d\cdot f&=f\cdot \tau_{i,\infty}^{\tau_{i,n}},\\
    \tau_{i+1,j+1}\cdot f&=f\cdot\tau_{i,j},\label{eq:ij}
  \end{align}
  for all $1\le i,j\le n-1$.
\end{lem}
\begin{proof}
  We only consider~\eqref{eq:1conj}, the other ones being checked in a similar but easier manner. Set $\phi=\tau_{1,\infty}^{\tau_{i+1},\infty}$; then, writing $\delta=\gamma_i\gamma_{i-1}\cdots\gamma_1$ and $\beta=\gamma_{i+1}^\delta\gamma_\infty$, we have
  \begin{align*}
    \phi(\gamma_1) &= \gamma_1^{\gamma_\infty^\beta},\\
    \phi(\gamma_{i+1}) &= \gamma_{i+1}^{\delta\gamma_\infty^{-\gamma_1^{\beta^{-1}}}\gamma_\infty\delta^{-1}},\\
    \phi(\gamma_\infty) &= \gamma_\infty^{\beta\gamma_1\gamma_\infty^{-\gamma_1^{\beta^{-1}}\gamma_\infty}\beta},\\
  \end{align*}
  all other generators being fixed. The biset $B(\phi\circ f)$ of the post-composition of $f$ with $\phi$ is then presented as follows in a basis $(y_1,\dots,y_d)$, with $\epsilon=\gamma_i\cdots\gamma_1$ and $\zeta=\gamma_{n-1}\cdots\gamma_{i+1}$ so $\gamma_\infty\gamma_n\zeta\gamma_i\epsilon=1$:
  \[\mathfrak B(\phi\circ f):\begin{cases}
    y_1\cdot\gamma_1 &= \gamma_i^{-\zeta^{-1}}\cdot y_2,\\
    y_{d-1}\cdot\gamma_1 &= \gamma_i^{-\zeta^{-1}\gamma_\infty}\cdot y_d,\\
  y_d\cdot\gamma_1 &= \gamma_\infty^{-1}\gamma_i^{\zeta^{-1}}\epsilon^{-1}\zeta^{-1} \cdot y_1,\\
    y_1\cdot\gamma_{i+1} &= \gamma_i^{\zeta^{-1}\epsilon^{-1}} \cdot y_1,\\
    y_1\cdot\gamma_\infty &= \zeta\epsilon\gamma_i^{-\zeta^{-1}}\gamma_\infty\cdot y_d,\\
    y_2\cdot\gamma_\infty &= \epsilon^{-1}\zeta^{-1}\cdot y_1,\\
    y_d\cdot\gamma_\infty &= \gamma_i^{\zeta^{-1}\gamma_\infty}\cdot y_2,
  \end{cases}\]
  all other entries being as in the presentation of $\mathfrak B(f)$. On the other hand, set $\psi=\tau_{i,\infty}\tau_{i,n}\tau_{n,\infty}$; we have
  \begin{align*}
    \psi(\gamma_i) &= \gamma_i^{\zeta^{-1}\epsilon^{-1}},\\
    \psi(\gamma_n) &= \gamma_n^{\epsilon^{-1}\zeta^{-1}},\\
    \psi(\gamma_\infty) &= \gamma_\infty^{\epsilon^{-1}\zeta^{-1}},
  \end{align*}
  all other generators being fixed. The biset $\mathfrak B(f\circ\psi)$ is presented as follows in a basis $(z_1,\dots,z_d)$:
  \[\mathfrak B(f\circ\psi):\begin{cases}
    x_1\cdot\gamma_1 &= (\gamma_\infty\gamma_n)^{\epsilon^{-1}\zeta^{-1}}\cdot x_2,\\
    x_d\cdot\gamma_1 &= \gamma_\infty^{-\epsilon^{-1}\zeta^{-1}}\cdot x_1,\\
    x_1\cdot\gamma_{i+1} &= \gamma_i^{\zeta^{-1}\epsilon^{-1}}\cdot x_1,\\
    x_1\cdot\gamma_\infty &= \gamma_\infty^{\epsilon^{-1}\zeta^{-1}}\cdot x_d,\\
  \end{cases}\]
  all other entries being as in the presentation of $\mathfrak B(f)$. Now to prove that $\mathfrak B(\phi\circ f)$ and $\mathfrak B(f\circ\psi)$ are isomorphic, it suffices to map the basis of the former into the latter, as follows:
  \begin{align*}
    y_1 &\mapsto z_1,\\
    y_k &\mapsto \epsilon^{-1}\zeta^{-1}\cdot z_k \text{ for }1<k<d,\\
    y_d &\mapsto \gamma_i^{\zeta^{-1}\gamma_\infty}\epsilon^{-1}\zeta^{-1}\cdot z_d,
  \end{align*}
  and to check that the right actions of $\pi$ on $\mathfrak B(\phi\circ f)$ and $\mathfrak B(f\circ\psi)$ are intertwined by this map.
\end{proof}

\begin{cor}\label{cor:conj}
  We have $W=P f^{\Mod(B)}$, for $P$ the group of point pushes and $\Mod(B)$ acting by conjugation.
\end{cor}
\begin{proof}
  The solution to the Hurwitz problem implies $W=\Mod(B) f^{\Mod(B)}$. Now consider the set $M\subseteq\Mod(B)$ of all $m\in\Mod(B)$ with $m f^{\Mod(B)}\subseteq P f^{\Mod(B)}$. Clearly $M$ is closed under conjugation, since $P$ is normal in $\Mod(B)$. Also, $M$ is a subgroup: if $m,n\in M$ and $m f^{\Mod(B)},n f^{\Mod(B)}\subseteq P f^{\Mod(B)}$, then
  \[m n f^{\Mod(B)}\subseteq m(P f^{\Mod(B)})=m f^{\Mod(B)} P\subseteq P f^{\Mod(B)}P= P^2 f^{\Mod(B)}= P f^{\Mod(B)}.\]
  Consider $m\in M$, and assume that we have an identity $m\cdot f=f\cdot h$ in $W$; then $h\cdot f=(f\cdot h)^{h^{-1}}=(m\cdot f)^{h^{-1}}=m^{h^{-1}}\cdot f^{h^{-1}}\in P f^{\Mod(B)}$; so $h\in M$.

  Obviously we have $P\subseteq M$, namely $\tau_{i,\infty}\in M$ for all $i$. Then, by Lemma~\ref{lem}\eqref{eq:1conj}, we have $\tau_{i,n}\in M$. Then Lemma~\ref{lem}\eqref{eq:ij} gives $\tau_{i-1,n-1}\in M$, etc., so finally all $\tau_{i,j}\in M$. Therefore $M=\Mod(B)$.
\end{proof}

\section{Proof of Theorem~\ref{thm:main}}
We expand on the sketch presented in~\S\ref{ss:sketch}. Consider first the map $\tau\colon\mathscr P_\Pi\ni(f,b,c)\mapsto(b,c)$ sending a branched covering to its marked sets $b,c$ on $S^2$. This map is evidently a fibration, with fibre $\mathscr W$ consisting of all branched coverings $f$ with fixed $b,c$. We treat, from now on, $A,B,C$ as fixed subsets of $S^2$ with $A\subset B\cap C$, so $\mathscr W$ is the set of branched coverings $(S^2,C)\to(S^2,B)$ that map $C$ to $B$ with combinatorics and degree prescribed by $\Pi$.

We let $W$ denote the quotient of $\mathscr W$ under isotopy. For any two branched coverings $f,f'\in\mathscr W$ that are isotopic, there is a unique homeomorphism $h\in\Homeo_0(S^2,C)$ with $f'=f\circ h$; in other words, $W$ is the quotient of $\mathscr W$ by the free action of $\Homeo_0(S^2,C)$. Now the group $\Homeo_0(S^2,C)$ is contractible, assuming $\#C\ge3$. It follows from the long exact sequence of a fibration and Whitehead's theorem that $\mathscr W$ has the homotopy type of a discrete set, and (except in dimension $\le1$) the homotopy type of $\mathscr P_\Pi$ is that of the base $\mathscr Q$ of the fibration,
\[\mathscr Q\coloneqq\{(b\colon B\hookrightarrow S^2,c\colon C\hookrightarrow S^2): b\restriction A=c\restriction A\}.\]
We have thus proven Proposition~\ref{prop:homotopy}, modulo the main connectedness claim.

There are two commuting actions on $W$, by the pure mapping class groups $\Mod(B)$ and $\Mod(C)$, respectively by post-composition and pre-composition. The action of $\Mod(C)$ is free with finitely many orbits. More precisely, for every $f\in W$ there is a finite-index subgroup $L_f\le\Mod(B)$ consisting of ``liftable'' classes: all $\ell\in\Mod(B)$ for which there exists $m\in\Mod(B)$ with $\ell\circ f\cong f\circ m$. We denote this (necessarily unique) $m\in\Mod(C)$ by $\sigma_f(\ell)$, defining thus a homomorphism $\sigma_f\colon L_f\to\Mod(C)$.

The long exact sequence of homotopy groups gives
\[\begin{tikzcd}\pi_1(\mathscr Q)\arrow[r] & \pi_0(\mathscr W)\arrow[r] & \pi_0(\mathscr P_\Pi)\arrow[r] & \pi_0(\mathscr Q).\end{tikzcd}\]
There is a natural map $\varepsilon_B\colon\Mod(B)\twoheadrightarrow\Mod(A)$ induced by the inclusion $A\subset B$; and similarly $\varepsilon_C\colon\Mod(C)\twoheadrightarrow\Mod(A)$. The equalizer of these two maps is the subgroup
\[E=\{(g,h)\in\Mod(B)\times\Mod(C):\varepsilon_B(g)=\varepsilon_C(h)\}\]
and we naturally have $E\cong\pi_1(\mathscr Q)$. We thus have
\begin{lem}
  The space $\mathscr P_\Pi$ is path connected if and only if the action of $E$ on $W$ is transitive.\qed
\end{lem}

\subsection{Specifics for degree-$d$ bicritical maps with an $n$-cycle}
All the above considerations applied to a general portrait $\Pi$. We now consider the specific choice of a portrait $\Pi$ specifying an $n$-cycle marked $A$ that contains a single order-$d$ critical point, and another order-$d$ critical point marked in $C$, with its image marked in $B$.

\begin{lem}
  The action of $\Mod(B)\times\Mod(C)$ on $W$ is transitive.
\end{lem}
\begin{proof}
  In considering the action of $\Mod(B)\times\Mod(C)$, we are in effect ignoring the condition that $A$ is an invariant $n$-cycle; that is, we are considering the space of coverings of $(S^2,B)$ with two order-$d$ critical values. Hurwitz's theory~\cite{hurwitz:ramifiedsurfaces} classifies its connected components as follows: there is then a bijection between, on the one hand, isomorphism classes of coverings with two order-$d$ critical values $b_1,b_2$; and on the other hand orbits of pairs $(\rho_1,\rho_2)$ of $d$-cycles in $\operatorname{Sym}(d)$ with $\rho_1\rho_2=1$, under the diagonal action of $\operatorname{Sym}(d)$ by conjugation. There is evidently a single such orbit, since all $d$-cycles are conjugate.
\end{proof}

(As mentioned in the sketch, this fact holds more generally for any portrait with a critical value of maximal order, by~\cite{thom:polynome}*{Th\'eor\`eme (1)}. Imposing that this critical value be fixed defines a polynomial slice.)

Let us now fix a polynomial map $f\in\mathscr P_{d,n}$, assuming $B=C$. We consider two subgroups of $E$: the diagonal $\Delta=\{(g,g):g\in\Mod(B)\}$, which acts on $W$ by conjugation, and $P=\ker(\varepsilon_B)$, the ``point pushes'' of the critical value, which act on $W$ by post-composition. We have $P=\pi_1(S^2\setminus A,b)$, and Birman's exact sequence~\cite{birman:ses} gives
\[\begin{tikzcd}1\arrow[r] & P\arrow[r] & \Mod(B)\arrow[r,"\varepsilon_B"] & \Mod(A)\arrow[r] & 1.\end{tikzcd}\]
It suffices to show that every element of $W$ may be written $p g f g^{-1}$ with $g\in\Mod(B)$ and $p\in P$. Now this is precisely Corollary~\ref{cor:conj}.


\end{document}